\documentclass[10pt]{amsart}
\usepackage{url}
\usepackage{hyperref}
\usepackage{enumerate}
\usepackage{comment}

\makeatletter

\@namedef{subjclassname@2010}{

  \textup{2010} Mathematics Subject Classification}

\usepackage{amssymb, amsmath,amsthm}
\usepackage{mathrsfs}
\newtheorem{thm}{Theorem}[section]

\newtheorem{prop}[thm]{Proposition}

\newtheorem{lem}[thm]{Lemma}
\theoremstyle{definition}

\newcommand{\ra}{\rightarrow}
\newcommand{\bk}{\backslash}
\newcommand{\mc}{\mathcal}

\newcommand{\mb}{\mathbb}
\newcommand{\sg}{\sigma}

\renewcommand{\ss}{\substack}
\newcommand{\llf}{\left\lfloor}

\newcommand{\e}{\varepsilon}
\newcommand{\rrf}{\right\rfloor}

\renewcommand{\bar}{\overline}

\begin{document}
\title{Sharp bounds for maximal sums of odd order Dirichlet characters}
\author{Alexander P. Mangerel}
\address{Department of Mathematical Sciences, Durham University, Stockton Road, Durham, DH1 3LE, UK}
\email{smangerel@gmail.com}

\begin{abstract}
Let $g \geq 3$ be fixed and odd, and for large $q$ let $\chi$ be a primitive Dirichlet character modulo $q$ of order $g$. Conditionally on GRH we improve the existing upper bounds in the P\'{o}lya-Vinogradov inequality for $\chi$, showing that
$$
M(\chi) := \max_{t \geq 1} \left|\sum_{n \leq t} \chi(n) \right| \ll \sqrt{q} \frac{(\log\log q)^{1-\delta_g} (\log\log\log\log\log q)^{\delta_g}}{(\log\log\log q)^{1/4}},
$$ 
where $\delta_g := 1-\tfrac{g}{\pi}\sin(\pi/g)$. Furthermore, we show unconditionally that there is an infinite sequence of order $g$ primitive characters $\chi_j$ modulo $q_j$ for which
$$
M(\chi_j) \gg \sqrt{q_j} \frac{(\log\log q_j)^{1-\delta_g} (\log\log\log\log\log q_j)^{\delta_g}}{(\log\log\log q_j)^{1/4}},
$$
so that our GRH bound is sharp up to the implicit constant. This improves on previous work of Granville and Soundararajan, of Goldmakher, and of Lamzouri and the author. 
\end{abstract}
\maketitle
\section{Introduction}
Let $q \geq 3$. Given a non-principal Dirichlet character $\chi$ modulo $q$ we define
$$
M(\chi) := \max_{t \geq 1} \left|\sum_{n \leq t} \chi(n)\right|.
$$
It is a classical problem to obtain sharp bounds for $M(\chi)$. The P\'{o}lya-Vinogradov inequality, proved over 100 years ago, yields the unconditional bound $M(\chi) \ll \sqrt{q}\log q$. Montgomery and Vaughan \cite{MV} showed in 1977 that this can be improved, assuming the Generalised Riemann Hypothesis (GRH), to 
$$
M(\chi) \ll \sqrt{q}\log\log q.
$$ 
The latter bound turns out to be the best possible for general non-principal Dirichlet characters, as Paley \cite{Pal} constructed an infinite sequence of quadratic characters $\chi_j$ modulo $q_j$ such that 
$$
M(\chi_j) \gg \sqrt{q_j} \log\log q_j.
$$ 
Besides the implicit constant therein, no substantial progress has been made towards improving the unconditional bound of P\'{o}lya and Vinogradov for \emph{general} non-principal characters. \\
Almost 20 years ago, Granville and Soundararajan \cite{GSPret} made a breakthrough when they showed that if $\chi$ is a primitive character of fixed \emph{odd} order then the unconditional and GRH bounds in the P\'{o}lya-Vinogradov inequality may both be improved. More precisely, let $g \geq 3$ be a fixed odd integer, and let $\chi$ be a primitive Dirichlet character modulo $q$ of order $g$. In the sequel, write 
$$
Q := \begin{cases} \log q &\text{ assuming GRH,} \\ q &\text{ unconditionally.}\end{cases}
$$
In \cite{GSPret} it was shown, using methods from pretentious number theory (that they themselves developed), that as $q \ra \infty$,
$$
M(\chi) \ll \sqrt{q} (\log Q)^{1-\delta_g/2+o(1)},
$$
where $\delta_g := 1-\tfrac{g}{\pi}\sin(\pi/g)>0$. This was subsequently improved by Goldmakher \cite{Gold}, who refined their method to replace the exponent $1-\delta_g/2$ with $1-\delta_g$. He also constructed, on GRH, an infinite sequence $\chi_j$ of primitive characters modulo $q_j$ of order $g$ such that for every $\e > 0$,
$$
M(\chi_j) \gg_{\e} \sqrt{q_j} (\log\log q_j)^{1-\delta_g - \e} \text{ as } j \ra \infty.
$$
Subsequently, Lamzouri and the author \cite{LamMan} developed new methods to improve upon Goldmakher's result, both conditionally and unconditionally. They showed in particular that on GRH there is a constant $C > 0$, depending at most on $g$, such that if $\chi$ is a primitive character modulo $q$ of order $g$ then
\begin{equation}\label{eq:LamManUpp}
M(\chi) \ll \sqrt{q}(\log\log q)^{1-\delta_g}(\log\log\log q)^{-1/4} (\log\log\log\log q)^C.
\end{equation}
Conversely, they showed unconditionally that there is an infinite sequence of primitive characters $\chi_j$ modulo $q_j$ of order $g$ such that
\begin{equation}\label{eq:LamManLow}
M(\chi_j) \gg \sqrt{q_j}(\log\log q_j)^{1-\delta_g}(\log\log\log q_j)^{-1/4} (\log\log\log\log q_j)^{-C},
\end{equation}
with the same constant $C$ as above. \\
Evidently, despite the improvements to previous work obtained in \cite{LamMan}, the sharp rate of growth of $M(\chi)$, for $\chi$ modulo $q$ of fixed odd order, as a function of $q$ remained unknown. This is in contrast to the situation for \emph{even} order characters, which appears to be rather better understood. Indeed, assuming GRH Granville and Soundararajan \cite{GSPret} generalised Paley's examples to produce, for every \emph{even} order $k$ an infinite family of primitive characters $\chi_j$ modulo $q_j$ of order $k$ such that 
$$
M(\chi_j) \gg \sqrt{q_j}\log\log q_j,
$$
a result that was then rendered unconditional by Goldmakher and Lamzouri \cite{LamGold} (see  \cite{Bob} for an alternative approach, as well as \cite{Lamk} for a more precise result). \\
%Thus, the optimal rate of growth for $M(\chi)$ with $\chi$ of fixed \emph{even} order has known. \\
In this note, we sharpen \eqref{eq:LamManUpp} conditionally on GRH, and obtain an unconditional construction of characters $\chi_j$ of order $g$ that improve upon \eqref{eq:LamManLow}.
%following sharpening of the work of \cite{LamMan}, 
We thereby obtain, for the first time, the sharp rate of growth of $M(\chi)$ as a function of $q$, for Dirichlet characters $\chi$ of fixed odd order $g$, conditionally on\footnote{Unfortunately, as our method does not circumvent the issues arising from Siegel zeros, we are unable to improve upon the unconditional Theorem 1.1 of \cite{LamMan}.} GRH.
\begin{thm}\label{thm:oddOrdMag}
Assume GRH. Let $g\geq 3$ be a fixed odd integer. Then for any primitive character $\chi$ modulo $q$ of order $g$ we have
$$
M(\chi) \ll \sqrt{q}\frac{(\log\log q)^{1-\delta_g} (\log\log\log\log\log q)^{\delta_g}}{(\log\log\log q)^{1/4}}.
$$
%where $\delta_g := 1-\tfrac{g}{\pi}\sin(\pi/g)$.
\end{thm}
%The next theorem shows, unconditionally, that this result is best possible.
\begin{thm}\label{thm:sharp}
%Assume GRH. 
Let $g \geq 3$ be a fixed odd integer. There are arbitrarily large $q$ and primitive Dirichlet characters $\chi$ modulo $q$ of order $g$ such that
$$
M(\chi) \gg \sqrt{q}\frac{(\log\log q)^{1-\delta_g} (\log\log\log\log\log q)^{\delta_g}}{(\log\log\log q)^{1/4}}.
%\sqrt{q}(\log\log q)^{1-\delta_g} (\log\log\log q)^{-1/4} (\log\log\log\log\log q)^{\delta_g}.
$$
\end{thm}
%\begin{rem}
%It is natural to speculate about the optimal implicit constant in Theorem \ref{thm:oddOrdMag}. Unfortunately, our methods are unable to furnish this, even on GRH, as we rely on a truncation of Euler products that gives rise to $O(1)$ error terms. 
%\end{rem}
\subsection{Proof Ideas}
In order to surpass the results in \cite{LamMan} we refine the methods of the latter paper to gain further insight into what is required of the extremal characters in the problem. With this insight we are able to quickly obtain Theorem \ref{thm:oddOrdMag}, and also improve upon the construction of the characters $\chi$ that yield \eqref{eq:LamManLow}, which leads to Theorem \ref{thm:sharp}. To describe more precisely our key new ideas, we shall recall some of the details of the work on odd order characters in \cite{GSPret}, and the work that succeeded it. \\
Let $\chi$ be a primitive character of large conductor $q$ and fixed odd order $g$. Granville and Soundararajan, and later Goldmakher, showed that the estimation of $M(\chi)$ may be (essentially) reduced to obtaining lower bounds for the prime sum
$$
\sum_{p \leq y} \frac{1-\text{Re}(\chi(p)\bar{\psi}(p)p^{-it})}{p},
$$
for each $(\log q)^A \leq y \leq q$, each $t \in \mb{R}$ satisfying $|t| \leq \log y$, and every odd\footnote{Recall that a Dirichlet character $\xi$ is called \emph{odd} if $\xi(-1) = -1$, and \emph{even} otherwise. We will use without mention the facts that if $\xi$ is odd then its order is even, and thus if $\xi$ has odd order then $\xi$ must be even.} primitive character $\psi$ of conductor $m \leq \log y$. For the purposes of this explanation we shall assume that $t = 0$. Letting $\mu_g$ denote the set of complex roots of unity of order $g$, a lower bound for this quantity may naturally be obtained by instead seeking an upper bound for
\begin{equation}\label{eq:maxCorrel}
\mc{S}(y;\psi,g) := \sum_{p \leq y} \frac{1}{p}\max_{\ss{z \in \mu_g \cup \{0\}}} \text{Re}(z \bar{\psi}(p)).
\end{equation}
Given $\psi$ modulo $m$ of order $k$, say, Granville and Soundararajan observed that since the conductor of $\psi$ is at most $\log y$ we may appeal to the Siegel-Walfisz theorem on primes in arithmetic progressions to show that\footnote{Given $t \in \mb{R}$ we write $e(t) := e^{2\pi it}$.}
%reduce the upper bound of this latter quantity to estimates for sums of the shape
$$
\mc{S}(y; \psi,g) = \left(\frac{1}{k} \sum_{\ell \pmod{k}} \max_{\ss{z \in \mu_g \cup \{0\}}} \text{Re}(z e(-\ell/k)) + o_{y \ra \infty}(1)\right) \log\log y.
$$
They then computed the bracketed expression and showed that it is always $< \delta_g$ (and in fact is asymptotically $\delta_g + o(1)$, if $k = k(y) \ra \infty$ as $y \ra \infty$). \\
One can go further and resolve the lower order terms in the above estimation of $\mc{S}(y;\psi,g)$, in particular by taking account of the size of $k$, and choosing $\psi$ optimally in order to produce the worst case upper bound. This is ultimately what enabled Lamzouri and the author to obtain a savings of size $(\log\log\log q)^{-1/4}$ in \eqref{eq:LamManUpp}, on GRH. \\
The shape of \eqref{eq:LamManUpp} is influenced by the types of constructions available to prove \eqref{eq:LamManLow}. In \cite[Sec. 4]{LamMan} it was shown that, given a choice of $\psi$ one can produce an order $g$ character $\chi$ for which, in some restricted range of small primes $p$, $z = \chi(p) \in \mu_g \cup \{0\}$ maximises $\text{Re}(z \bar{\psi}(p))$. By choosing $m$ to be prime and $\psi$ to be a character of maximal order $k = m-1$ (thus necessarily odd), one could then produce a character $\chi$ for which \eqref{eq:maxCorrel} differs from its worst case upper by an error term of size $O(\log\log m)$. By choosing $m$ of the optimal size  $m \asymp \sqrt{\log\log\log q}$, one arrives at the bound \eqref{eq:LamManUpp}, with the unresolved $(\log\log\log\log q)^C$ term corresponding to the aforementioned $O(\log\log m)$ error. \\
A short-coming of the above argument is that it does not attempt to control $\psi$ itself any further than in controlling the size of its modulus $m$, and forcing its order to be as large as possible. As a consequence of Proposition \ref{prop:Spsig} below, which estimates $\mc{S}(y;\psi,g)$ up to $O(1)$ error terms (under a hypothesis on $\psi$ that is weaker than GRH, and can be guaranteed in the proof of Theorem \ref{thm:sharp}), we surmise that the extremal $\psi$ for \eqref{eq:maxCorrel} has the property that (roughly speaking) $\psi(p) = 1 + o(1)$ uniformly over all primes $p \ll \log m$. Thus, to improve upon \eqref{eq:LamManLow} we must ensure that we can find such an \emph{odd} character $\psi$ of \emph{large} order ($k \gg m$, rather than $k = m-1$, being sufficient) for which we may construct $\chi$ of order $g$ as described above. \\
To carry out this construction, in Section \ref{subsec:BujApp} we use a nice recent result of Bujold \cite{Buj}, showing that for \emph{most} prime moduli $m$ there are \emph{many} characters $\xi$ with $\xi(-1) = -1$, such that $\xi(p) = 1 + o(1)$ uniformly at all $p \ll \log m$. By choosing large primes $m$ judiciously so that $m-1$ has few small prime factors, we show that at least one of the odd characters $\xi \pmod{m}$, which we choose to be $\psi$, has order $k \geq (m-1)/2$. Feeding this character $\psi$ into the methods of \cite[Sec. 4]{LamMan}, we produce an order $g$ character $\chi$ for which the estimate of Theorem \ref{thm:sharp} holds.

\section{Sharpening estimates for distances}
Throughout this section, we fix an odd integer $g \geq 3$. Let $y \geq 3$, and let $\psi$ be an odd primitive Dirichlet character modulo $m \leq (\log y)^{4/7}$ of order $k$. Our interest in this section is to give a precise estimate for the sum
\begin{align*}
\mc{S}(y;\psi,g) := \sum_{p \leq y} \frac{1}{p}\max_{z_p \in \mu_g \cup \{0\}} \text{Re}(z_p \bar{\psi}(p)),
\end{align*}
recalling that $\mu_g$ denotes the set of complex roots of unity of order $g$. To this end we will prove Proposition \ref{prop:Spsig}. To state it, we require some further notation and definitions. \\
In the sequel, we write $k^{\ast} := k/(g,k)$ and $g^\ast := g/(g,k)$. We also set $\tilde{\psi} := \psi^{(k,g)}$. Note that as $g$ is odd and $k$ is even, $k^\ast \geq 2$. Finally, for $\eta \in (0,1/2)$ we say that a modulus $m$ is \emph{$\eta$-good} if the product of Dirichlet $L$-functions
$$
\prod_{\xi \neq \xi_0 \pmod{m}} L(s,\xi)
$$ 
has no zeroes in the rectangle
$$
1-\eta < \text{Re}(s) \leq 1, \quad |\text{Im}(s)| \leq (\log m)^2.
$$
Obviously, under the assumption of GRH every modulus $m$ is $\eta$-good, for any $\eta \in (0,1/2)$.
\begin{prop} \label{prop:Spsig}
Assume there is some $\eta \in (0,1/2)$ such that $m$ is $\eta$-good. Then\footnote{Given $t \in \mb{R}$ we write $\|t\| := \min_{k \in \mb{Z}} |t-k|$.} 
\begin{align*}
\mc{S}(y;\psi,g) &= (1-\delta_g)\frac{\pi/(gk^\ast)}{\tan(\pi/(gk^\ast))} \log\left(\frac{\log y}{\log\log m}\right) +\sum_{\ell \pmod{k^\ast}} \cos(\tfrac{2\pi}{g}\|g^\ast \ell/k^\ast\|)\left(\sum_{\ss{p \leq (\log m)/100 \\ \tilde{\psi}(p) = e(\ell/k^\ast)}} \frac{1}{p}\right)
%-\frac{1}{k^\ast} \sum_{\ss{p \leq (\log m)/100 \\ p \nmid m}} \frac{1}{p}\right) 
+ O_{\eta}(1).
\end{align*}
\end{prop}
%\sum_{\ell \pmod{k}} \text{Re}(z_\ell e(-\ell/k)) \sum_{\ss{p \leq y \\ \psi(p) = e(\ell}
\subsection{Identifying the lower order terms in $\mc{S}(y;\psi,g)$}
\noindent For each $\ell \pmod{k}$ let us fix a choice of $z_{\ell} \in \mu_g \cup \{0\}$ such that 
\begin{equation}\label{eq:zEllDef}
\max_{z  \in \mu_g \cup \{0\}} \text{Re}(z e(-\ell/k)) = \text{Re}(z_\ell e(-\ell/k)).
\end{equation}
\begin{lem} \label{lem:zEll}
For each $\ell \pmod{k}$ we have $\text{Re}(z_\ell e(-\ell/k)) = \cos(\tfrac{2\pi}{g} \|g^\ast\ell/k^\ast\|)$, and moreover
\begin{equation}\label{eq:deltgId}
\frac{1}{k}\sum_{\ell \pmod{k}} \text{Re}(z_\ell e(-\ell/k)) = (1-\delta_g)\frac{\pi/(gk^\ast)}{\tan(\pi/(gk^\ast))}.
\end{equation}
\end{lem}
\begin{proof}
We will show that we may choose $z_{\ell}^\ast \in \mu_g$ that maximises the real parts, such that $\text{Re}(z_\ell^\ast e(-\ell/k)) > 0$. Therefore, $z_\ell = z_\ell^\ast$. Given $0 \leq \ell < k$, let $0 \leq n_{\ell} \leq g$ be chosen such that $z_{\ell}^\ast = e(n_\ell/g)$. Then we have
$$
\text{Re}(z_\ell^\ast e(-\ell/k)) = \cos\left(\tfrac{2\pi}{g}\left|n_\ell - \frac{\ell g}{k}\right|\right).
$$
Choosing $n_\ell$ to be the closest integer to $\ell g/k = \ell g^\ast/k^\ast$, we get by definition that
$$
\left|n_\ell - \frac{\ell g}{k}\right| = \|\ell g^\ast/k^\ast\|.
$$
As $g \geq 3$, $\tfrac{2\pi}{g} \|\ell g^\ast/k^\ast\| \in [0,\pi/3]$, so that $z_\ell^\ast e(-\ell/k)$ must have positive real part. Thus, $z_\ell = z_{\ell}^\ast$ as asserted. The formula \eqref{eq:deltgId} is given in \cite[Lem. 8.4]{GSPret}.
\end{proof}
Applying Lemma \ref{lem:zEll} and using the definition of $k$ and the periodicity of $\psi$ modulo $m$, we may write
\begin{align}\label{eq:expinAPs}
\mc{S}(y;\psi,g) = \sum_{\ell \pmod{k}} \text{Re}(z_\ell e(-\ell/k)) \sum_{\ss{p \leq y \\ \psi(p) = e(\ell/k)}} \frac{1}{p} = \sum_{\ell \pmod{k}} \cos(\tfrac{2\pi}{g} \|g^\ast\ell/k^\ast\|) \sum_{\ss{a \pmod{m} \\ \psi(a) = e(\ell/k)}} \left(\sum_{\ss{p \leq y \\ p \equiv a \pmod{m}}} \frac{1}{p}\right).
\end{align}
By a result of Languasco and Zaccagnini \cite[Thm. 2, Cor. 3]{LanZac}, we have that for every $a \pmod{m}$ coprime to $m$,
\begin{equation} \label{eq:LangZac}
\sum_{\ss{p \leq y \\ p \equiv a \pmod{m}}} \frac{1}{p} = \frac{1}{\phi(m)} \log\log y - C_m(a) + O\left(\sum_{\ss{p \leq y \\ p \equiv a \pmod{m}}} \frac{1}{p^2} + \frac{(\log\log y)^{16/5}}{(\log y)^{3/5}}\right),
\end{equation}
where we define
\begin{equation}\label{eq:CmaDef}
C_m(a) := \frac{1}{\phi(m)} \sum_{\xi \neq \xi_0 \pmod{m}} \bar{\xi}(a) \log \frac{K(1,\xi)}{L(1,\xi)} - \frac{1}{\phi(m)}(\gamma + \log(\phi(m)/m)).
\end{equation}
%In the sequel, we write $\tilde{\psi} := \psi^{(k,g)}$, and 
Here, for a primitive character $\xi$ we have introduced the Dirichlet series 
$$
K(s,\xi) := \sum_{n \geq 1} \frac{k_{\xi}(n)}{n^s},
$$
absolutely convergent for $\text{Re}(s) > 0$, where $k_{\xi}(n)$ is the completely multiplicative function defined at primes $p$ via
\begin{equation}\label{eq:kxiDef}
k_{\xi}(p) := p\left(1-\left(1-\frac{\xi(p)}{p}\right)\left(1-\frac{1}{p}\right)^{-\xi(p)}\right).
\end{equation}
\begin{lem} \label{lem:CmaFixl}
For each $\ell \pmod{k}$,
$$
\sum_{\ss{a \pmod{m} \\ \psi(a) = e(\ell/k)}} C_m(a) = \frac{1}{k} \sum_{\ss{j \pmod{k} \\ j \neq 0}} e\left(-\frac{\ell j }{k}\right) \log \frac{K(1,\psi^j)}{L(1,\psi^j)} - \frac{1}{k} \left(\gamma + \log(\phi(m)/m)\right).
$$
\end{lem}
\begin{proof}
Note first of all that $\{a \pmod{m} : \psi(a) = e(\ell/k)\}$ is a coset of the kernel $\{a \pmod{m} : \psi(a) = 1\}$, and since the range of $\psi$ has size $k$ we get by basic group theory that
\begin{equation}\label{eq:fixedPsiCard}
|\{a \pmod{m} : \psi(a) = e(\ell/k)\}| = \phi(m)/k.
\end{equation}
It therefore follows from \eqref{eq:CmaDef} that
$$
R_\ell:= \sum_{\ss{a \pmod{m} \\ \psi(a) = e(\ell/k)}} C_m(a) = \frac{1}{\phi(m)} \sum_{\xi \neq \xi_0 \pmod{m}} \log\frac{K(1,\xi)}{L(1,\xi)} \sum_{\ss{a \pmod{m} \\ \psi(a) = e(\ell/k)}} \bar{\xi}(a) - \frac{1}{k}(\gamma + \log(\phi(m)/m)). 
$$
Now, for each $(a,m) = 1$ we have by orthogonality of additive characters that
$$
1_{\psi(a) = e(\ell/k)} = \frac{1}{k} \sum_{j \pmod{k}} \psi^j(a) e(-j\ell/k).
$$
Inserting this into the previous expression and using orthogonality of Dirichlet characters modulo $m$, 
%we have
\begin{align*}
R_\ell &= \frac{1}{\phi(m)} \sum_{\xi \neq \xi_0 \pmod{m}} \log\frac{K(1,\xi)}{L(1,\xi)} \cdot \frac{1}{k} \sum_{j \pmod{k}} e(-\ell j /k) \sum_{\ss{a \pmod{m}}} \psi^j(a) \bar{\xi}(a) - \frac{1}{k}(\gamma + \log(\phi(m)/m))\\
&= \sum_{\xi \neq \xi_0 \pmod{m}} \log\frac{K(1,\xi)}{L(1,\xi)} \cdot \frac{1}{k} \sum_{j \pmod{k}} e(-\ell j/k) 1_{\xi = \psi^j} - \frac{1}{k}(\gamma + \log(\phi(m)/m))\\
&= \frac{1}{k} \sum_{\ss{j \pmod{k} \\ j \neq 0}} e(-\ell j /k) \log \frac{K(1,\psi^j)}{L(1,\psi^j)} - \frac{1}{k}(\gamma + \log(\phi(m)/m)),
\end{align*}
as claimed.
\end{proof}
\noindent In the sequel, for $j \pmod{k^\ast}$ we define
$$
S_j := \frac{1}{k^\ast} \sum_{\ell \pmod{k^\ast}} \cos(\tfrac{2\pi}{g}\|g^\ast\ell/k^\ast\|)e\left(-\frac{j\ell}{k}\right).
$$
Applying Lemma \ref{lem:CmaFixl}, multiplying the sum over $a$ therein by $\cos(\tfrac{2\pi}{g}\|g^\ast\ell/k^\ast\|)$ and summing over $\ell \pmod{k}$, we see that
\begin{align}
\sum_{\ell \pmod{k}} \cos(\tfrac{2\pi}{g}\|g^\ast \ell/k^\ast\|) \sum_{\ss{a \pmod m \\ \psi(a) = e(\ell/k)}} C_m(a) 
&= \sum_{\ss{j \pmod{k} \\ j \neq 0}} \log\frac{K(1,\psi^j)}{L(1,\psi^j)}  \left(\frac{1}{k} \sum_{\ell \pmod{k}} \cos(\tfrac{2\pi}{g} \|g^\ast \ell/k^\ast\|) e\left(-\frac{j\ell}{k}\right)\right) \label{eq:passToSj} \\
&-(\gamma + \log(\phi(m)/m)) \left(\frac{1}{k} \sum_{\ell \pmod{k}} \cos(\tfrac{2\pi}{g} \|g^\ast \ell/k^\ast\|)\right). \label{eq:S0term}
\end{align}
Note that $\|g^\ast \ell/k^\ast\|$ depends only on the residue class of $\ell$ modulo $k^\ast$. Writing $\ell = vk^\ast + u$ for $0 \leq u < k^\ast$ and $0 \leq v < (k,g)$, and summing over both $u$ and $v$, we deduce that for each $j \pmod{k}$,
$$
\sum_{\ell \pmod{k}} \cos(\tfrac{2\pi}{g} \|g^\ast \ell/k^\ast\|) e(-\ell j/k) = (g,k) 1_{(g,k)|j} \sum_{u \pmod{k^\ast}} \cos(\tfrac{2\pi}{g} \|g^\ast u/k^\ast\|) e(-u \tfrac{j/(g,k)}{k^\ast}).
$$
Replacing $j \pmod{k}$ by $j/(k,g) \pmod{k^\ast}$ whenever the latter is an integer in \eqref{eq:passToSj}, and recalling the definition of $\tilde{\psi}$ and $S_j$, the left-hand side of \eqref{eq:passToSj} becomes
$$
%\sum_{\ell \pmod{k}} \cos(\tfrac{2\pi}{g}\|g^\ast \ell/k^\ast\|) \sum_{\ss{a \pmod m \\ \psi(a) = e(\ell/k)}} C_m(a) = 
\sum_{\ss{j \pmod{k^\ast} \\ j \neq 0}} S_j\log\frac{K(1,\tilde{\psi}^j)}{L(1,\tilde{\psi}^j)} -  (1-\delta_g)\frac{\pi/(gk^\ast)}{\tan(\pi/(gk^\ast))}(\gamma + \log(\phi(m)/m)),
$$
%where we 
%recall that $\tilde{\psi} = \psi^{(k,g)}$, and 
using Lemma \ref{lem:zEll} to evaluate the expression in \eqref{eq:S0term}. \\
Next, combining this with \eqref{eq:expinAPs}, \eqref{eq:LangZac}, \eqref{eq:fixedPsiCard} and Lemma \ref{lem:zEll}, and recalling that $m \leq (\log y)^{4/7}$, we obtain
\begin{align}
&\mc{S}(y;\psi,g)
%\sum_{\ell \pmod{k}} \text{Re}(z_\ell e(-\ell/k) \sum_{\ss{a \pmod{m} \\ \psi(a) = e(\ell/k)}} \sum_{\ss{p \leq y \\ p \equiv a \pmod{m}}} \frac{1}{p} \\
= \sum_{\ell \pmod{k}} \cos(\tfrac{2\pi}{g} \|g^\ast \ell/k^\ast\|)\left(\frac{1}{k} \log\log y - \sum_{\ss{a \pmod{m} \\ \psi(a) = e(\ell/k)}} C_m(a) + O\left(\frac{\phi(m)(\log\log y)^{16/5}}{k(\log y)^{3/5}}\right)\right) + O(1) \nonumber\\
%&= (1-\delta_g)\frac{\pi/(gk^\ast)}{\tan(\pi/(gk^\ast))}\left(\log\log y + \gamma + \log(\phi(m)/m) + O\left(\frac{m (\log\log y)^{16/5}}{(\log y)^{3/5}}\right)\right) - \sum_{\ss{j \pmod{k} \\ j \neq 0}}S_j \log \frac{K(1,\tilde{\psi}^j)}{L(1,\tilde{\psi}^j)} \nonumber \\
&= (1-\delta_g)\frac{\pi/(gk^\ast)}{\tan(\pi/(gk^\ast))}\left(\log\log y + \gamma + \log(\phi(m)/m)\right) - \sum_{\ss{j \pmod{k} \\ j \neq 0}}S_j \log \frac{K(1,\tilde{\psi}^j)}{L(1,\tilde{\psi}^j)} + O(1). \label{eq:withSj} 
\end{align}
%Since $\phi(m) \leq m \leq (\log y)^{4/7}$, when summed over $\ell \pmod{k}$ the error terms (crudely) contribute $O(1)$. \\
Our remaining task is now to estimate the sum over $j$ in \eqref{eq:withSj}.
\subsection{Estimates for $S_j$ and their consequences}
In order to control error terms that arise from approximating $\log(K(1,\tilde{\psi}^j)/L(1,\tilde{\psi}^j))$ by short sums over primes in Lemma \ref{lem:controlErr} below, we first establish the following $L^1$ bound for $|S_j|$.
\begin{lem} \label{lem:SjBds}
We have
$$
%\max_{\ss{j \pmod{k^\ast} \\ j \neq 0}} |S_j| = O(1), \quad 
\sum_{\ss{j \pmod{k^\ast} \\ j \neq 0}} |S_j| \ll \log k.
$$
\end{lem}
\begin{proof}
We write $S_j = (S_j^+ + S_j^-)/2$, where
$$
S_j^{\pm} := \frac{1}{k^\ast} \sum_{\ell \pmod{k^\ast}} e\left(\pm \frac{1}{g} \left\| \frac{g^\ast\ell}{k^\ast}\right\| - \frac{\ell j}{k^\ast}\right).
$$
Let $1 \leq a < k^\ast$ be such that $ag^\ast \equiv 1 \pmod{k^\ast}$; as $g$ is odd and $k^\ast$ is even, $a$ must be odd. Making the change of variables $\ell \mapsto g^\ast \ell$, then reparametrising, we get 
$$
S_j^{\pm} = \frac{1}{k^\ast} \sum_{-k^\ast/2 < w \leq k^\ast/2} e\left(\pm \frac{1}{g} \left\|\frac{w}{k^\ast}\right\| - \frac{a jw}{k^\ast}\right) = \frac{1}{k^\ast} \sum_{-k^\ast/2 < w \leq k^\ast/2} e\left(\pm \frac{|w|}{gk^\ast} - \frac{a j w}{k^\ast}\right).
$$
Summing each of $S_j^\pm$ over $w$, we obtain
\begin{align*}
S_j^\pm  &= \frac{1}{k^\ast} \sum_{0 \leq w' < k^\ast/2} e\left(\pm \frac{w'}{gk^\ast}\right)\left(e\left(- \frac{w'aj}{k^\ast}\right) + e\left(\frac{w'aj}{k^\ast}\right)\right) + O\left(\frac{1}{k^\ast}\right)\\
&= \frac{(-1)^j e(\pm \tfrac{1}{2g})-1}{k^\ast} \left(\frac{1}{e(\tfrac{1}{gk^\ast}(\pm 1- ajg)) - 1} + \frac{1}{e(\tfrac{1}{gk^\ast}(\pm 1+ ajg)) - 1}\right) + O\left(\frac{1}{k^\ast}\right).
%&= \pm (-1)^j e\left(\frac{1}{2gk^\ast }(\pm 1 - ajg)\right) \frac{\sin(\pi/g)}{k^\ast \sin(\tfrac{\pi}{gk^\ast}(\pm 1 - ajg))}.
%e\left(\frac{1}{2g}(\pm 1-ajg)\right) \frac{1}{k^\ast} \sum_{0 \leq w' < k^\ast} e\left(-\frac{w'}{gk^\ast}(\pm 1 - ajg)\right) = \frac{(-1)^j e(\pm \tfrac{1}{2g})}{k^\ast} \frac{e(\mp \tfrac{1}{g})-1}{e(-\tfrac{1}{gk^\ast}(\pm 1- ajg)) - 1} \\
%&= \pm (-1)^j e\left(\frac{1}{2gk^\ast }(\pm 1 - ajg)\right) \frac{\sin(\pi/g)}{k^\ast \sin(\tfrac{\pi}{gk^\ast}(\pm 1 - ajg))}.
%e\left(\frac{1}{2g}(\pm 1-ajg)\right) \frac{1}{k^\ast} \sum_{0 \leq w' < k^\ast} e\left(-\frac{w'}{gk^\ast}(\pm 1 - ajg)\right) = \frac{(-1)^j e(\pm \tfrac{1}{2g})}{k^\ast} \frac{e(\mp \tfrac{1}{g})-1}{e(-\tfrac{1}{gk^\ast}(\pm 1- ajg)) - 1} \\
%&= \pm (-1)^j e\left(\frac{1}{2gk^\ast }(\pm 1 - ajg)\right) \frac{\sin(\pi/g)}{k^\ast \sin(\tfrac{\pi}{gk^\ast}(\pm 1 - ajg))}.
\end{align*}
%Noting that by definition, $|S_j| \leq 1$, 
Applying the triangle inequality, we deduce that for each non-zero $j \pmod{k^\ast}$
$$
|S_j| \ll \frac{1}{k^\ast}\left(\frac{1}{|\sin(\tfrac{\pi}{gk^\ast}(1-ajg))|} + \frac{1}{|\sin(\tfrac{\pi}{gk^\ast}(1+ajg))|} + 1\right).
$$
Noting that $g \geq 3$, using the simple bound $|\sin(\pi \theta)| \geq 2\|\theta\|$, summing over $j \neq 0$ and making the invertible change of variables $j \mapsto aj \pmod{k^\ast}$, we get
\begin{align*}
\sum_{\ss{j \pmod{k^\ast} \\ j \neq 0}} |S_j| &\ll \frac{1}{k^\ast}\sum_{\sg \in \{-1,+1\}} \, \sum_{\ss{j \pmod{k^\ast} \\ j \neq 0}} \frac{1}{\|\tfrac{1}{gk^\ast} + \sg \tfrac{aj}{k^\ast}\|} + 1 &\ll \frac{1}{k^\ast}\sum_{\ss{1 \leq j' \leq k^\ast/2}} \frac{1}{|j'/k^\ast - 1/(gk^\ast)|} + 1\\
&\ll \sum_{\ss{1 \leq j' \leq k^\ast/2}} \frac{1}{j'} + 1 \ll \log k,
\end{align*}
as required.
\end{proof}
\begin{lem} \label{lem:controlErr}
Assume that there is some $\eta \in (0,1/2)$ such that $m$ is $\eta$-good. Then
$$
\sum_{\ss{j \pmod{k^\ast} \\ j \neq 0}} S_j \log \frac{K(1,\tilde{\psi}^j)}{L(1,\tilde{\psi}^j)} = -\sum_{\ell \pmod{k^\ast}} \cos(\tfrac{2\pi}{g} \|g^\ast \ell/k^\ast\|) \left(\sum_{\ss{p \leq (\log m)/100 \\ \tilde{\psi}(p) = e(\ell/k^\ast)}} \frac{1}{p} - \frac{1}{k^\ast} \sum_{\ss{p \leq (\log m)/100 \\ p \nmid m}} \frac{1}{p} \right) + O_{\eta}(1).
%\sum_{p \leq (\log m)^B} \left( \sum_{\ss{j \pmod{k^\ast} \\ j \neq 0}}  S_j \tilde{\psi}(p)^j\right) \log(1-1/p) + O\left(\frac{1}{(\log m)^A}\right).
$$
\end{lem}
\begin{proof}
For the moment, fix a non-zero residue class $j$ modulo $k^\ast$, and for convenience set $\xi := \tilde{\psi}^j$. Rearranging \eqref{eq:kxiDef}, we find that for every $p \nmid m$,
\begin{equation}\label{eq:combKL}
-\log\left(1-\frac{k_{\xi}(p)}{p}\right) + \log\left(1-\frac{\xi(p)}{p}\right) = \xi(p) \log\left(1-\frac{1}{p}\right).
\end{equation}
As discussed in \cite[Sec. 4]{LamMan}, for every $X\geq 3$ we have
$$
\log K(1,\xi) = -\sum_{p \leq X} \log\left(1-\frac{k_{\xi}(p)}{p}\right) + O\left(\frac{1}{X}\right).
$$
Moreover, since $m$ is $\eta$-good and each $\xi$ is non-principal,
%$\tilde{\psi}^j$ whenever $j \neq 0$, 
taking $X = (\log m)^{5/\eta}$, $T = (\log m)^{2}$ and $\sg_0 = 1-\eta$ in \cite[Lem. 3.4]{LamMan} we get
$$
\log L(1,\xi) = -\sum_{p\leq (\log m)^{5/\eta}} \log\left(1-\frac{\xi(p)}{p}\right) + O\left(\frac{1}{\log m}\right).
$$
Combining these latter two estimates with \eqref{eq:combKL}, we get
$$
\log\frac{K(1,\tilde{\psi}^j)}{L(1,\tilde{\psi}^j)} = \sum_{p \leq (\log m)^{5/\eta}} \tilde{\psi}^j(p) \log\left(1-\frac{1}{p}\right) + O\left(\frac{1}{\log m}\right).
$$
Multiplying by $S_j$, summing over $j \neq 0$ and applying Lemma \ref{lem:SjBds} (noting that $k \leq m$), we obtain
\begin{align*}
\sum_{\ss{j \pmod{k^\ast} \\ j \neq 0}} S_j \log \frac{K(1,\tilde{\psi}^j)}{L(1,\tilde{\psi}^j)} &= \sum_{p \leq (\log m)^{5/\eta}} \left(\sum_{\ss{j \pmod{k^\ast} \\ j \neq 0}} S_j \tilde{\psi}(p)^j\right) \log\left(1-\frac{1}{p}\right) + O\left(\frac{1}{\log m}\sum_{\ss{j \pmod{k^\ast} \\ j \neq 0}} |S_j|\right) \\
&= \sum_{p \leq (\log m)^{5/\eta}} \left(\sum_{\ss{j \pmod{k^\ast} \\ j \neq 0}} S_j \tilde{\psi}(p)^j\right) \log\left(1-\frac{1}{p}\right) + O(1).
\end{align*}
%In the sequel, 
Let $\mc{M}$ denote the main term in the last expression. Let us momentarily fix $p \leq (\log m)^{5/\eta}$, $p \nmid m$. Recalling the definition of $S_j$, we get
\begin{align*}
\sum_{\ss{j \pmod{k^\ast} \\ j \neq 0}} S_j \tilde{\psi}(p)^j &= \sum_{\ell \pmod{k^\ast}} \cos(\tfrac{2\pi}{g}\|g^\ast \ell /k^\ast\|) \frac{1}{k^\ast} \left(\sum_{j \pmod{k^\ast}} (\tilde{\psi}(p) e(-\ell/k^\ast))^j - 1\right) \\
&= \sum_{\ell \pmod{k^\ast}} \cos(\tfrac{2\pi}{g}\|g^\ast \ell /k^\ast\|)\left(1_{\tilde{\psi}(p) = e(\ell/k^\ast)} - \frac{1}{k^\ast}\right).
\end{align*}
Summing over all $p \leq (\log m)^{5/\eta}$ with $p \nmid m$, then using Mertens' theorem to handle the range
%\footnote{This truncation appears to prevent us from obtaining an optimal implicit constant in our main theorems, even if we assume GRH.} 
$(\log m)/100 < p \leq (\log m)^{5/\eta}$ in both sums, we get
\begin{align*}
\mc{M} &= \sum_{\ell \pmod{k^\ast}} \cos(\tfrac{2\pi}{g}\|g^\ast \ell /k^\ast\|)\left(\sum_{\ss{p \leq (\log m)^{5/\eta} \\ \tilde{\psi}(p) = e(\ell/k^\ast)}} \log\left(1-\frac{1}{p}\right) - \frac{1}{k^\ast} \sum_{\ss{p \leq (\log m)^{5/\eta} \\ p \nmid m}} \log\left(1-\frac{1}{p}\right)\right) \\
&= -\sum_{\ell \pmod{k^\ast}} \cos(\tfrac{2\pi}{g}\|g^\ast \ell /k^\ast\|)\left(\sum_{\ss{p \leq (\log m)/100 \\ \tilde{\psi}(p) = e(\ell/k^\ast)}} \frac{1}{p} - \frac{1}{k^\ast} \sum_{\ss{p \leq (\log m)/100 \\ p \nmid m}} \frac{1}{p} \right) + O_{\eta}(1).
\end{align*}
This implies the claim.
\end{proof}
\begin{proof}[Proof of Proposition \ref{prop:Spsig}]
Combining Lemma \ref{lem:controlErr} with \eqref{eq:withSj}, noting that the contribution from $\gamma$ to the sum is $O(1)$, we get
% we get
\begin{align*}
\mc{S}(y;\psi,g) &= (1-\delta_g)\frac{\pi/(gk^\ast)}{\tan(\pi/(gk^\ast))}\left(\log\log y - \log(m/\phi(m))\right) \\
&+\sum_{\ell \pmod{k^\ast}} \cos(\tfrac{2\pi}{g} \|g^\ast \ell/k^\ast\|) \left(\sum_{\ss{p \leq (\log m)/100 \\ \tilde{\psi}(p) = e(\ell/k^\ast)}} \frac{1}{p} - \frac{1}{k^\ast} \sum_{\ss{p \leq (\log m)/100 \\ p \nmid m}} \frac{1}{p} \right) + O_{\eta}(1),
\end{align*}
Now, by Mertens' theorem and the crude bound $\omega(m) \ll \log m$, we obtain
\begin{align}
\sum_{\ss{p \leq (\log m)/100 \\ p \nmid m}} \frac{1}{p} &= \log\log\log m - \sum_{\ss{p|m \\ p \leq (\log m)/100}} \frac{1}{p} + O(1) \nonumber\\
%&= \log\log\log m - \log(m/\phi(m)) + O(1+\omega(m)/\log m) \\
&= \log\log\log m - \log(m/\phi(m)) + O(1). \label{eq:pDivm}
\end{align}
Inserting this into the previous expression for $\mc{S}(y;\psi,g)$, then recalling \eqref{eq:deltgId} and rearranging, we obtain the proposition.
%as claimed.
\end{proof}
\section{Relating the upper bound for $M(\chi)$ to $\mc{S}(y;\psi,g)$}
In this section we recall several key results arising in \cite{LamMan} that allow us to connect our work on $\mc{S}(y;\psi,g)$ to the proof of Theorem \ref{thm:oddOrdMag}.\\
%To begin with, let us recall the key inputs to the bounds in both directions. 
Let $q \geq 3$, and in the sequel we write 
%\begin{align*}
$Q := \log q$.
%:= \begin{cases} (\log q)^{12} &\text{ on GRH} \\ q &\text{ unconditionally.}
%\end{cases}
%\end{align*}
Put $\mb{U} := \{z \in \mb{C} : |z| \leq 1\}$, and for $x \geq 1$ and multiplicative functions $f,g : \mb{N} \ra \mb{U}$, define
$$
\mb{D}(f,g;x) := \left(\sum_{p \leq x} \frac{1-\text{Re}(f(p)\bar{g}(p))}{p}\right)^{1/2}.
$$
Also, if $T >0$
%$x \geq 3$ and $f: \mb{N} \ra \mb{U}$ is multiplicative 
then we write
$$
\mc{M}(f;x,T) := \min_{|t| \leq T} \mb{D}(f,n^{it};x)^2. 
$$
%Throughout, we will assume GRH unless otherwise indicated. 
\begin{thm}[\cite{LamMan}, Thm 2.3] \label{thm:LMUpp}
Assume GRH. Let $\chi$ be a primitive character modulo $q$. Let $\psi$ be the odd primitive character of conductor below $(\log Q)^{4/7}$ for which $\psi \mapsto \mc{M}(\chi\bar{\psi};Q, (\log Q)^{-7/11})$ is minimised. Then
$$
M(\chi) \ll \frac{\sqrt{qm}}{\phi(m)}(\log Q) \exp\left(-\mc{M}(\chi\bar{\psi}; Q,(\log Q)^{-7/11})\right) + \sqrt{q}(\log Q)^{9/11 + o(1)}.
$$
\end{thm}
%\begin{proof}
%This follows straightforwardly from the proof of \cite[Thm. 2.5]{LamMan}, the key point being that $L(1,\c\cite[(3.1)]{LamMan} holds for some $A = A(\eta) > 0$ when $m$ is $\eta$-good. We leave the details to the reader.
%\end{proof}
\begin{prop}[\cite{LamMan}, Prop. 5.1]  \label{prop:MxT}
Let $\chi$ be a primitive character modulo $q$ of odd order $g \geq 3$ and let $\psi$ be an odd primitive character modulo $m \leq (\log Q)^{4/7}$. Let $t \in \mb{R}$, $|t| \leq (\log Q)^{-7/11}$, and set $z := \exp((\log Q)^{7/11})$. Then
$$
\sum_{z < p \leq Q} \frac{1-\text{Re}(\chi(p)\bar{\psi}(p)p^{-it})}{p} \geq \delta_g \log\left(\frac{\log Q}{\log z}\right) + O(1).
$$
\end{prop}
\begin{proof}[Proof of Theorem \ref{thm:oddOrdMag}]
Let $\chi$ be a primitive character modulo $q$ of odd order $g \geq 3$. Assuming GRH and applying Theorem \ref{thm:LMUpp}, we can find an odd primitive character $\psi$ of conductor $m \leq (\log Q)^{4/7}$ for which
\begin{align} \label{eq:MchiUppfirst}
M(\chi) \ll \frac{\sqrt{qm}}{\phi(m)}(\log Q) \exp\left(-\mc{M}(\chi\bar{\psi}; Q,(\log Q)^{-7/11})\right) + \sqrt{q}(\log Q)^{9/11 + o(1)}.
\end{align}
We may assume henceforth that there is a $9/11< c < 1-\delta_g$ such that $M(\chi) \gg \sqrt{q} (\log Q)^c$, since otherwise Theorem \ref{thm:oddOrdMag} follows as soon as $q$ is sufficiently large. Thus, by increasing the implicit constant as needed, when $q$ is large enough we obtain from \eqref{eq:MchiUppfirst} the bound
\begin{equation}\label{eq:MchiUpp}
M(\chi) \ll \frac{\sqrt{qm}}{\phi(m)}(\log Q) \exp\left(-\mc{M}(\chi\bar{\psi}; Q,(\log Q)^{-7/11})\right).
% + \sqrt{q}(\log Q)^{9/11 + o(1)}.
\end{equation}
Set $\alpha := 7/11$, so that $z = \exp((\log Q)^\alpha)$. By a standard application of Taylor expansion and Mertens' theorems, it is easy to show that as $|t| \leq \tfrac{1}{\log z}$,
$$
\mb{D}(\chi,\psi n^{it};z)^2 = \mb{D}(\chi,\psi; z)^2 + O(1).
$$
Thus, by Proposition \ref{prop:MxT} we have
\begin{equation}\label{eq:McmBd}
\mc{M}(\chi\bar{\psi}; Q,(\log Q)^{-\alpha}) \geq (1-\alpha) \delta_g\log\log Q + \mb{D}(\chi,\psi; z)^2 + O(1).
\end{equation}
Furthermore, it is clear that
$$
\mb{D}(\chi,\psi;z)^2 = \log\log z - \sum_{p \leq z} \frac{\text{Re}(\chi(p)\bar{\psi}(p))}{p} +O(1) \geq \alpha \log\log Q - \mc{S}(z;\psi,g) + O(1).
$$
For convenience, in the sequel we write $G(x) := x/\tan x$ whenever $x \neq \pi r$ for some $r \in \mb{Z}$. By Proposition \ref{prop:Spsig} (which applies with any $\eta \in (0,1/2)$ as we are assuming GRH), we thus obtain
\begin{align*}
\mb{D}(\chi,\psi;z)^2 &\geq  \left(1-(1-\delta_g)G(\pi/(gk^\ast))\right) \alpha\log\log Q + (1-\delta_g)G(\pi/(gk^{\ast}))\log\log\log m \\
&-\sum_{\ell \pmod{k^\ast}} \cos(\tfrac{2\pi}{g}\|g^\ast \ell/k^\ast\|)\left(\sum_{\ss{p \leq (\log m)/100 \\ \tilde{\psi}(p) = e(\ell/k^\ast)}} \frac{1}{p}\right) + O(1).
\end{align*}
Applying the trivial upper bound $\cos(\tfrac{2\pi}{g}\|g^\ast \ell/k^\ast\|) \leq 1$ for each $\ell \pmod{k^\ast}$, we thus obtain by \eqref{eq:pDivm} that
$$
\mb{D}(\chi,\psi;z)^2 \geq \left(1-(1-\delta_g)G(\pi/(gk^\ast))\right) (\alpha \log\log Q- \log\log\log m) + \log(m/\phi(m)) + O(1). 
$$
Inserting this bound into \eqref{eq:McmBd} and using the Taylor expansion $G(x) = 1 - cx^2 + O(x^4)$ with some $c > 0$ (whose precise value will be irrelevant), valid for all $0 < x \leq \pi/6$ (and recalling that $gk^\ast = g^\ast k \geq 6$), it follows that
\begin{align*}
\mc{M}(\chi\bar{\psi}; Q,(\log Q)^{-\alpha}) &\geq \delta_g (\log\log Q- \log\log\log m) + \frac{\alpha c\pi^2(1-\delta_g)}{(gk^\ast)^2} \log\log Q + \log(m/\phi(m)) \\
&+ O\left(1+\frac{\log\log Q}{k^4} + \frac{\log\log\log m}{k^2}\right).
\end{align*}
Inserting this lower bound into \eqref{eq:MchiUpp}, we see that
\begin{align*}
\log\left(\frac{M(\chi)}{\sqrt{q}}\right) &\leq (1-\delta_g) \log\log Q  + \delta_g \log\log\log m - \left(\frac{1}{2} \log m + \frac{\alpha c\pi^2(1-\delta_g)}{(gk^\ast)^2}\log\log Q\right) \\
&+ O\left(1 + \frac{\log\log Q}{k^4} + \frac{\log\log\log m}{k^2}\right).
\end{align*}
Setting $c_1 := \alpha c \pi^2 (1-\delta_g)/(g^{\ast})^2$, we see that 
$$
\frac{1}{2} \log m + \frac{\alpha c\pi^2(1-\delta_g)}{(gk^\ast)^2}\log\log Q \geq \frac{1}{2} \log \phi(m) + \frac{c_1 \log\log Q}{\phi(m)^2},
$$
a bound which is attained (up to $O(1)$ error) whenever $m$ is prime and $k = \phi(m) = m-1$.  Moreover, the above expression is minimised when $m = \sqrt{4c_1\log\log Q} + O(1)$. Thus, we see that at this optimal point (with $m = \sqrt{4c_1\log\log Q} + O(1)$ prime, $k = \phi(m)$),
$$
\log\left(\frac{M(\chi)}{\sqrt{q}}\right) \leq  (1-\delta_g) \log\log Q - \frac{1}{4} \log\log \log Q + \delta_g \log\log\log \log\log Q + O(1),
$$
and the claim follows on exponentiating.
\end{proof}
\section{Relating the lower bound on $M(\chi)$ to $\mc{S}(y;\psi,g)$}
In this section we will prove Theorem \ref{thm:sharp}. In this direction, the following result from \cite{LamMan} is relevant.
\begin{thm}[special case of \cite{LamMan}, Thm. 2.5] \label{thm:LMLow}
Let $N$ be large and let $m \leq \log N$.
% be such that there is $\eta \in (0,1/2)$ for which $m$ is $\eta$-good. 
Let $\psi$ be an odd primitive character modulo $m$. Then for all but at most $O(N^{1/4})$ primitive characters $\chi$ of conductor $q \leq N$ we have
\begin{equation}\label{eq:lowBdMChi}
M(\chi) + \sqrt{q} \gg \frac{\sqrt{qm}}{\phi(m)} (\log Q) \exp(-\mb{D}(\chi,\psi;Q)^2).
\end{equation}
\end{thm}
In view of Theorem \ref{thm:LMLow}, the proof of Theorem \ref{thm:sharp} will reduce to obtaining an asymptotic, up to an $O(1)$ error term, for $\mb{D}(\chi,\psi;Q)^2$, where $\chi$ is a suitable primitive character modulo $q$ of odd order $g$, and $\psi$ is a suitable odd primitive character modulo $m \leq \log Q$. To match the bound in Theorem \ref{thm:oddOrdMag}, this amounts to ensuring that, up to $O(1)$ error terms, the lower bounds and optimisation argument present in the proof of Theorem \ref{thm:oddOrdMag} can be turned into equalities. \\
Reviewing the proof of Theorem \ref{thm:oddOrdMag}, it is apparently sufficient for us to show is that there are $\eta,\delta \in (0,1/2)$ for which the following data is available: 
\begin{enumerate}
\item an infinite sequence of prime moduli $m_j$ that are $\eta$-good;
\item an \emph{odd} primitive character $\psi_j$ of order $\gg m_j$ for each $m_j$, such that for most $p \leq (\log m_j)/100$, if $\psi(p) = e(\ell/k)$ then $\cos(\tfrac{2\pi}{g}\|g^\ast \ell/k^\ast\|)$ is close to $1$;
\item a scale $N$ such that, $m_j \asymp \sqrt{\log\log N}$; and
\item $\gg N^{\delta}$ primitive characters $\chi_j$ modulo $q_j \leq N$ of order $g$ for which, setting $Q_j := \log q_j$ as before,
$$
\mb{D}(\chi_j,\psi_j;Q_j)^2 = \log\log Q_j - \mc{S}(Q_j;\psi_j,g) + O_{\eta}(1).
%(1-\delta_g) \log\log Q_j - \frac{1}{4} \log\log \log Q_j + \delta_g \log\log\log \log\log Q_j + O(1).
$$
\end{enumerate}
Items (1) and (2) will dictate how we select $\psi_j$, whereas items (3) and (4) identify how we will select $\chi_j$ given $\psi_j$. \\
The construction of $\chi_j$ given $\psi_j$ is already present in \cite[Sec. 4]{LamMan}. Proposition 2.6, proved there, shows that if $y \leq \log N/10$ and $m \leq (\log y)^{4/7}$ is \emph{any} non-exceptional modulus (in the sense of Siegel zeros) and $\psi$ is any primitive character modulo $m$ of some order $k$ then one may construct $\gg \sqrt{N}$ appropriate order $g$ characters $\chi$ of conductors $q \leq N$ such that 
$$
\mb{D}(\chi,\psi;y)^2 = \left(1-(1-\delta_g)\frac{\pi/(gk^\ast)}{\tan(\pi/(gk^\ast))}\right) \log\log y + O(\log\log m).
$$
In light of our refined Proposition \ref{prop:Spsig}, on preselecting $\psi$ judiciously we shall ensure that a more precise bound is available. We shall prove the following proposition to this end.
\begin{prop}\label{prop:consChiPsi}
Let $g \geq 3$ be a fixed odd integer. Let $M$ be large. Then there exist
\begin{itemize}
\item a prime $m \in (M/2,M]$ that is $\tfrac{1}{10}$-good and for which $(m-1,g) = 1$,
\item an odd primitive character $\psi \pmod{m}$ of order $k \geq (m-1)/2$,
\item a modulus $q$ such that $m \asymp \sqrt{\log\log \log q}$, and 
\item a primitive character $\chi$ modulo $q$ of order $g$
\end{itemize}
such that \eqref{eq:lowBdMChi} holds, and if $Q := \log q$ then
$$
\mb{D}(\chi,\psi; Q)^2 = \left(1-(1-\delta_g) \frac{\pi/(gk)}{\tan(\pi/(gk))}\right) (\log\log Q - \log\log \log m) + \log(m/\phi(m)) + O(1).  
$$
\end{prop}
We first show that Proposition \ref{prop:consChiPsi} implies Theorem \ref{thm:sharp}.
\begin{proof}[Proof of Theorem \ref{thm:sharp} assuming Proposition \ref{prop:consChiPsi}]
For each $j \geq j_0$ set $M_j := 2^{j+1}$. By Proposition \ref{prop:consChiPsi} we may find a prime $m_j \in (M_j/2,M_j]$ with $(m_j-1,g) = 1$ that is a $\tfrac{1}{10}$-good modulus; an odd primitive character $\psi_j \pmod{m_j}$ with order $k_j \geq (m_j-1)/2$; a modulus $q_j$ with $m_j \asymp \sqrt{\log\log\log q_j}$; and a primitive character $\chi_j$ modulo $q_j$ of order $g$ such that \eqref{eq:lowBdMChi} holds, as does
$$
\mb{D}(\chi_j,\psi_j; Q_j)^2 = \left(1-(1-\delta_g) \frac{\pi/(gk_j)}{\tan(\pi/(gk_j))}\right) (\log\log Q_j - \log\log \log m_j) + \log(m_j/\phi(m_j)) + O(1).
$$
By Theorem \ref{thm:LMLow}, we have
$$
\log\left(\frac{M(\chi_j)}{\sqrt{q_j}}  + 1\right) \geq \log\log Q_j - \mb{D}(\chi_j,\psi_j;Q_j)^2 + \log(m_j/\phi(m_j)) - \frac{1}{2} \log m_j + O(1).
$$
Inserting our estimate for $\mb{D}(\chi_j,\psi_j;Q_j)^2$ and using the Taylor expansion $x/\tan x = 1-cx^2 + O(x^4)$ with $c > 0$ for $0 < x \leq \pi/6$ as before, the right-hand side of this last estimate is
$$
(1-\delta_g) \log\log Q_j + \delta_g \log\log\log m_j - \frac{1}{2}\log m_j - \frac{c\pi^2(1-\delta_g)}{gk_j^2} \log\log Q_j + O\left(1 + \frac{\log\log Q_j}{k_j^4} + \frac{\log\log\log m_j}{k_j^2}\right).
$$
As $(m_j-1)/2 \leq k_j \leq m_j-1$ and $m_j \asymp \sqrt{\log\log Q_j}$, we get that
$$
\log\left(\frac{M(\chi_j)}{\sqrt{q_j}}  + 1\right) \geq (1-\delta_g) \log\log Q_j - \frac{1}{4} \log\log \log Q_j + \delta_g\log\log\log\log\log Q_j + O(1).
$$
If $M_j$ (and thus $Q_j$) is sufficiently large then we deduce that $M(\chi_j)/\sqrt{q_j} \geq 10$, say, and the claim then follows upon exponentiating.
\end{proof}
\subsection{Constructing ``good'' moduli $m$ and odd characters $\psi$} \label{subsec:BujApp}
Our key ingredient in proving Proposition \ref{prop:consChiPsi} is the following. \begin{prop} \label{prop:consChar}
Let $g\geq 3$ be an odd integer and let $M \geq M_0$. Then there are: 
\begin{enumerate}[(i)]
\item a prime $m \in (M/2,M]$ that is $\tfrac{1}{10}$-good, and 
\item an odd primitive character $\psi \pmod{m}$ of order $k \geq (m-1)/2$ with $(k,g) = 1$, 
\end{enumerate}
such that\footnote{For $\alpha \in \mb{C} \bk \{0\}$ we $\text{arg}(\alpha) \in (-1/2,1/2]$ to be such that $\alpha = |\alpha|e(\text{arg}(\alpha)))$.}
$$
\max_{p \leq (\log m)/100} |\text{arg}(\psi(p))| \ll \frac{1}{\log \log M}.
$$
\end{prop}
This result is a consequence of the following theorem of Bujold (see \cite[Prop. 4.5]{Buj}. 
\begin{thm}[Bujold] \label{thm:Buj}
Let $M$ be a large integer, let $T \leq \frac{\log M}{100}$ and $N \leq \frac{T}{2(\log T)^3}$. Then for all but $O(M^{1/10})$ primes $m \leq M$ we have
$$
\left|\left\{\psi \pmod{m} : \,  \psi(-1) = -1, \, \max_{p \leq T} |\psi(p)-1| \leq \frac{1}{N}\right\}\right| \gg m \exp\left(-\frac{2T \log N}{\log T}\right).
$$
\end{thm}
To use this, we will construct primes $m$ for which Theorem \ref{thm:Buj} holds, and for which there are ``few'' characters $\xi$ of order $<(m-1)/2$. Naturally, as the order of a character modulo a prime $m$ is a divisor of $m-1$, it suffices to constrain the size of the odd prime factors of $m-1$, as we do below. In the sequel, for $n > 1$ we write $P^-(n)$ to denote the least prime factor of $n$.
\begin{lem} \label{lem:siftShift}
There is a $\delta \in (0,1/2)$ such that if $M$ is sufficiently large then there are $\gg\frac{M}{(\log M)^2}$ primes $m \in (M/2,M]$ such that $2||(m-1)$, $P^-((m-1)/2) > M^\delta$, and $m$ is a $\tfrac{1}{10}$-good modulus.
% and $P^+((m-1)/2) < M^{1-\delta}$.
\end{lem}
\begin{proof}
This is a straightforward consequence of the fundamental lemma of the sieve, the Bombieri-Vinogradov theorem and zero-density estimates for Dirichlet $L$-functions. \\ Precisely, let $s \geq 1$ be a parameter to be chosen later, and define
$$
\mc{A} := \{p-1 : p \in (M/2,M] \cap \mb{P}, p \equiv 3 \pmod{4}\}, \, D := M^{1/3}, \, z := D^{1/s}.
$$ 
Write $P(z) := \prod_{3 \leq p < z} p$, and let $(\lambda_d^{\pm})_{\ss{d|P(z) \\ d \leq D}}$ be the upper and lower bound sieve weights constructed in \cite[Lem. 6.3]{IK}, thus satisfying $|\lambda_d^\pm| \leq 1$ for all $d$. We may assume that $\lambda_d$ is supported on odd $d \leq D$, so we omit this redundant condition. \\ 
For each odd $d \leq D$ write
\begin{align*}
E((M/2,M];4d) &:= \max_{\ss{a \pmod{4d} \\ (a,4d) = 1}} \left|\sum_{\ss{M/2 < p \leq M \\ p \equiv a \pmod{4d}}} 1  - \frac{\pi(M)-\pi(M/2)}{\phi(4d)}\right|.
%&= \max_{\ss{a \pmod{4d} \\ (a,4d) = 1}} \left|\sum_{\ss{M/2 < p \leq M \\ p \equiv a \pmod{4d}}} 1 - \frac{|\mc{A}|}{\phi(d)}\right|.
\end{align*}
Employing the upper and lower bound sieves, we get
\begin{align*}
|\{a \in \mc{A} : (a,P(z)) = 1\}| &\leq \frac{\pi(M)-\pi(M/2)}{2} \sum_{\ss{d|P(z)}} \frac{\lambda_d^+}{\phi(d)} + O\left(\sum_{d \leq D} E((M/2,M];4d)\right) \\ 
|\{a \in \mc{A} : (a,P(z)) = 1\}| &\geq \frac{\pi(M)-\pi(M/2)}{2}\sum_{\ss{d|P(z)}} \frac{\lambda_d^-}{\phi(d)} - O\left(\sum_{d \leq D} E((M/2,M];4d)\right). 
\end{align*}
By the Bombieri-Vinogradov theorem (see e.g. \cite[Thm. 17.1]{IK}), the error terms yield
$$
\sum_{d \leq D} E((M/2,M];4d) \ll \frac{M}{(\log M)^{100}},
$$
say. %for any $\e > 0$.
Furthermore, by the fundamental lemma of the sieve \cite[Lem. 6.3]{IK} and the prime number theorem, the main terms in each of the upper and lower bounds are
$$
\left(1+O\left(e^{-s} + \frac{1}{\log M}\right)\right)\frac{M}{4\log M} \prod_{3 \leq p < z} \left(1-\frac{1}{p-1}\right).
$$
Finally, by standard zero-density estimates (e.g., \cite[(10.7)]{IK}, taking $k = 1$ and $\alpha = 9/10$), the number of moduli $m \in (M/2,M]$ for which $\prod_{\xi \neq \xi_0 \pmod{m}} L(z,\xi)$ vanishes in the rectangle
$$
9/10 \leq \text{Re}(z) \leq 1, \quad |\text{Im}(z)| \leq (\log M)^2
$$ 
is $\ll M^{1/5+o(1)}$. \\
Selecting $s$ and $M$ to be sufficiently large (in an absolute sense), using Mertens' theorem we get that there is an absolute constant $C_1 > 0$ such that 
$$
|\{a \in \mc{A} : (a,P(z)) = 1, \, a+1 \text{ is  $\tfrac{1}{10}$-good}\}| \geq \left(C_1s + o(1)\right)\frac{M}{(\log M)^2}.
$$
%as claimed.
By construction, the latter elements $a \in \mc{A}$ are of the form $a = m-1$, where $m$ is prime, $2||(m-1)$ and if $p|(m-1)$ with $p > 2$ then $p \geq z = M^{1/(3s)}$.
%% and $q < M^{1-\delta}$. 
We select $\delta = 1/(4s) > 0$, so that $P^-((m-1)/2) > M^{\delta}$, and the claim follows. 
%\\ In particular, all but $O(M^{1/5+o(1)})$ of the $\gg M/(\log M)^2$ elements in $\mc{A}$ with $(a,P(z)) = 1$ must be $\tfrac{1}{10}$-good, as required.
\end{proof}
\begin{lem}\label{lem:largeOrd}
Let $m \in (M/2,M]$ be a prime for which $2||(m-1)$ and $P^-((m-1)/2) > M^{\delta}$ for some $\delta > 0$. Then the number of characters $\psi \pmod{m}$ of order $< (m-1)/2$ is $O_{\delta}(M^{1-\delta})$. 
\end{lem}
\begin{proof}
Let $\xi$ be a generator for the group of Dirichlet characters modulo $m$. Then every  $\psi \pmod{m}$ is of the form $\psi = \xi^d$, for some $1 \leq d \leq m-1$, and the order of $\psi$ is $\text{ord}(\psi) = (m-1)/(d,m-1)$. By assumption, we may write $m-1 = 2p_1\cdots p_R$, where $R \ll \delta^{-1}$ and $M^{\delta} < p_1 \leq \cdots \leq p_R \leq m/2$ for all $1 \leq r \leq R$. Thus, $\text{ord}(\psi) < (m-1)/2$ if, and only if, $\psi = \xi^d$ such that there is $1 \leq r \leq R$ with $p_r|d$. But then
$$
|\{\psi \pmod{m} : \, \text{ord}(\psi) < (m-1)/2\}| \leq \sum_{1 \leq r \leq R} \sum_{\ss{1 \leq d \leq m-1 \\ p_r|d}} 1 \ll m \sum_{1 \leq r \leq R} \frac{1}{p_r} \leq \frac{RM}{p_1} \ll_{\delta} M^{1-\delta},
$$
as claimed.
\end{proof}
\begin{proof}[Proof of Proposition \ref{prop:consChar}]
Let $M$ be large. By Lemma \ref{lem:siftShift} there is a $\delta \in (0,1/2)$ and $\gg M/(\log M)^2$ primes $m \in (M/2,M]$ such that $m$ is $\tfrac{1}{10}$-good, $2||(m-1)$, and if $p|(m-1)$ with $p > 2$ then $p > M^{\delta}$. According to Lemma \ref{lem:largeOrd}, for each of these $m$ at most $O_{\delta}(M^{1-\delta})$ of the characters $\psi \pmod{m}$ have order $< (m-1)/2$. \\
If we now apply Theorem \ref{thm:Buj} (taking $N = \log \log M$ and $T = (\log M)/100$), we see that for all but $O(M^{1/10})$ primes $m' \in (M/2,M]$ there are 
$$
\gg M\exp\left(-\frac{2T(\log N)}{\log T}\right) \gg M \exp\left(-\frac{(\log M)(\log\log \log M)}{50\log\log M}\right)
$$
odd characters $\psi \pmod{m'}$ for which we have
\begin{equation} \label{eq:closeTo1}
|\psi(p)-1| \leq \frac{1}{\log\log M} \text{ whenever } p \leq (\log M)/100.
\end{equation}
Combining these two results, we see that for large enough $M$ we may select $m$ in Theorem \ref{thm:Buj} such that $m$ is a $\tfrac{1}{10}$-good modulus and there is an odd primitive character $\psi$ modulo $m$ satisfying \eqref{eq:closeTo1} with $k = \text{ord}(\psi) \geq (m-1)/2$. As $g$ is a fixed integer, if $M$ is large enough then $(g,m-1) =1$, and since $k|(m-1)$, $(k,g) = 1$ as well. Moreover, we clearly have $|\text{arg}(\psi(p))| \asymp |\psi(p)-1|$. The claim now follows.
\end{proof}
\begin{proof}[Proof of Proposition \ref{prop:consChiPsi}]
We first show that for large $M$ there is a character $\psi$ that satisfies the first two bullet points of the proposition, and such that whenever $y \geq e^{7m/4}$,
\begin{equation}\label{eq:Goal1}
\log\log y - \mc{S}(y;\psi,g) = \left(1-(1-\delta_g)\frac{\pi/(gk)}{\tan(\pi/(gk))}\right)\left(\log\log y - \log\log\log m\right) + \log(m/\phi(m)) + O(1).
\end{equation}
%whenever $y \geq e^{7m/4}$. \\
By Proposition \ref{prop:consChar}, there is a $\tfrac{1}{10}$-good modulus $m \in (M/2,M]$ and an odd primitive character $\psi \pmod{m}$ of order $k \geq (m-1)/2$, $(k,g) = 1$, such that $|\text{arg}(\psi(p))| \ll 1/(\log\log M)$ for all $p \leq (\log m)/100$. We therefore see that if $p \leq (\log m)/100$ then
$$
\psi(p) = e(\ell/k) \Rightarrow \|g\ell/k\| \leq g\|\ell/k\| \ll \frac{g}{\log \log M}.
$$
It follows by Taylor expansion that $\cos(\tfrac{2\pi}{g} \|g\ell/k\|) = 1 + O\left(\frac{1}{(\log \log M)^{2}}\right)$, and on combining this with \eqref{eq:pDivm}, 
%applying Proposition \ref{prop:Spsig} for this $\psi$ and $y \geq e^{7m/4}$, we find
%it suffices to show (recalling that $(k,g) = 1$) that
\begin{align*}
\sum_{\ell \pmod{k}} \cos(\tfrac{2\pi}{g}\|g \ell/k\|)\left(\sum_{\ss{p \leq (\log m)/100 \\ \tilde{\psi}(p) = e(\ell/k)}} \frac{1}{p}\right)  
&= \sum_{\ell \pmod{k}} \sum_{\ss{p \leq (\log m)/100 \\ \tilde{\psi}(p) = e(\ell/k)}} \frac{1}{p} + O\left(\frac{1}{\log \log M}\right) \\
&= \log\log\log m - \log(m/\phi(m)) + O(1).
\end{align*}
Applying Proposition \ref{prop:Spsig} for this $\psi$ and $y \geq e^{7m/4}$, we obtain \eqref{eq:Goal1}. \\
Next, we show that, given this $\psi$ we may choose $q$ with $m \asymp \sqrt{\log\log\log q}$ and a primitive character $\chi \pmod{q}$ of order $g$ such that \eqref{eq:lowBdMChi} holds, and 
$$
\mb{D}(\chi,\psi;\log q)^2 = \log\log\log q - \mc{S}(\log q; \psi,g) + O(1).
$$
Let $N$ be such that $m = \llf \sqrt{\log\log \log N} \rrf$. By \cite[Lem. 4.7]{LamMan} there are $\gg \sqrt{N}$ moduli $N^{1/3} < q \leq N$ and primitive characters $\chi$ modulo $q$ of order $g$ such that for every $p \leq (\log N)/100$, $p \nmid gm$, if $\psi(p) = e(\ell/k)$ then $\chi(p) = z_\ell$, with $z_\ell$ defined as in \eqref{eq:zEllDef}. Since $\log \log N = \log \log q + O(1)$, we easily obtain $m \asymp \sqrt{\log\log\log q}$, and also
\begin{align}
\mb{D}(\chi,\psi;\log q)^2 &= \log\log\log q - \sum_{\ell \pmod{k}} \text{Re}(z_\ell e(-\ell/k)) \sum_{\ss{p \leq (\log N)/100 \\ \psi(p) = e(\ell/k)}} \frac{1}{p} + O(1) \nonumber \\
&= \log\log\log q - \mc{S}(\log q; \psi,g) + O(1). \label{eq:Goal2}
\end{align}
As \eqref{eq:lowBdMChi} holds for all but $O(N^{1/4})$ characters $\chi$, some such $q$ and $\chi$ must satisfy both \eqref{eq:lowBdMChi} and \eqref{eq:Goal2} simultaneously. Combining this with \eqref{eq:Goal1} when $y = \log q$ (noting that indeed, $y\geq e^{7m/4}$), the proposition follows.
\end{proof}

\section*{Acknowledgments}
\noindent We thank the Lit and Phil library in Newcastle-upon-Tyne for excellent working conditions during the writing of this paper.

\bibliographystyle{plain}
\bibliography{sharpOddV3.bib}
%it suffices to show (recalling that $(k,g) = 1$) that
%where the last equality comes from \eqref{eq:pDivm}.
\end{document}